\documentclass{amsart}

\usepackage{mathtools}

\DeclarePairedDelimiter\abs{\lvert}{\rvert}
\DeclarePairedDelimiter\norm{\lVert}{\rVert}  % norm of a vector, etc.

\let\oldforall\forall
\renewcommand{\forall}{\oldforall \, }

\let\oldexist\exists
\renewcommand{\exists}{\oldexist \: }

\providecommand\given{}
\newcommand\SetSymbol[1][]{%
\nonscript\:#1\vert
\allowbreak
\nonscript\:
\mathopen{}}
\DeclarePairedDelimiterX\Set[1]\{\}{%
\renewcommand\given{\SetSymbol[\delimsize]}
#1
}

\DeclarePairedDelimiterXPP\lnorm[1]{}\lVert\rVert{_2}{#1}

\DeclarePairedDelimiterX{\Parn}[1](){#1} % parenthesis usable with \given

    \newcommand{\tr}{\operatorname{tr}}

    \newcommand{\hodge}{{*}}

        \newcommand{\del}{\partial}
        
        \newcommand{\vol}{\mathrm{vol}}

        \newcommand{\R}{\mathbb{R}}
        \newcommand{\C}{\mathbb{C}}

        \newcommand{\mc}[1]{\mathcal{#1}}

        \renewcommand{\epsilon}{\varepsilon}

\usepackage{enumitem}

\usepackage[dvipsnames]{xcolor}
\usepackage[colorlinks]{hyperref}
\hypersetup{
    linkcolor=BrickRed,
    citecolor=Green,
    filecolor=Mulberry,
    urlcolor=NavyBlue,
    menucolor=BrickRed,
    runcolor=Mulberry
    }
    
\usepackage[
        noabbrev,
        capitalise,
        nameinlink,
    ]
    {cleveref}
\crefname{lem}{Lemma}{Lemma}
\crefname{prop}{Proposition}{Proposition}

\usepackage{todonotes}

\newtheoremstyle{bfnoteonly}%
{}{}%
{\itshape}{}%
{\bfseries}{.}%
{ }%
{\thmnote{#3}}

\theoremstyle{bfnoteonly}

\theoremstyle{plain}
\newtheorem{thm}{Theorem}[section]
\newtheorem*{thm*}{Theorem}
\newtheorem{prop}[thm]{Proposition}
\newtheorem{lem}[thm]{Lemma}

\newtheorem{Question}[thm]{Question}

\theoremstyle{definition}
\newtheorem{defn}[thm]{Definition}

\newtheorem{rem}[thm]{Remark}
\newtheorem*{rem*}{Remark}

\theoremstyle{remark}

\numberwithin{equation}{section}

\usepackage{diffcoeff}[=v4]
\diffdef {}
{
op-order-sep = 0 mu
}

\begin{document}

%%
%% The title of the paper goes here.  Edit to your title.
%%

\title[Generalized Monge-Amp\`{e}re equation]{On a generalized Monge-Amp\`{e}re equation on closed almost K\"{a}hler surfaces}

%%
%% Now edit the following to give your name and address:
%% 

\author{Ken Wang}
\thanks{Supported by NSFC Grants 11971112  }

\address{School of Mathematical Sciences,
Fudan University,
Shanghai 100433,
China}
\email{kanwang22@m.fudan.edu.cn}

%%
%% If there is another author uncomment and edit the following.
%%

\author{Zuyi Zhang}
\thanks{}

\address{Beijing International Center for Mathematical Research,
China}
\email{zhangzuyi1993@pku.edu.cn}

%%
%% If there are three of more authors they are added in the obvious
%% way. 
%%

%%
%% If there is another author uncomment and edit the following.
%%

\author{Tao Zheng}
\thanks{Supported by NSFC Grants 12371078}

\address{School of Mathematics and Statistics, Beijing Institute of Technology, Beijing 100081,
China}
\email{zhengtao08@amss.ac.cn}

%%
%% If there are three of more authors they are added in the obvious
%% way. 
%%

\author{Peng Zhu}
\thanks{Supported by NSFC Grants 12171417}
\address{School of Mathematics and Physics, Jiangsu University of Technology, Changzhou, Jiangsu 213001, China}
\email{zhupeng@jsut.edu.cn}

%%%
%%% The following is for the keywords.  The keywords are optional and
%%% if not used just delete, or comment out, the following.
%%%


%%%
%%% The following is for the subjectclasses.  The subjectclasses are optional and
%%% if not used just delete, or comment out, the following.
%%%


%%%
%%% The following is for the abstract.  The abstract is optional and
%%% if not used just delete, or comment out, the following.
%%%

%\dedicatory{Dedicated to Professor Kung-Ching Chang on the occasion of his 90th birthday.}

\begin{abstract}
We show the uniqueness of the solution to a generalized Monge-Amp\`{e}re equation on closed almost K\"{a}hler surfaces, where the equation depends only on the underlying almost Kähler structure. Moreover, we obtain a result as Chu-Tosatti-Weinkove~\cite{ChuTW19}.  
\end{abstract}

%%
%%  LaTeX will not make the title for the paper unless told to do so.
%%  This is done by uncommenting the following.
%%

\maketitle
%%
%% LaTeX can automatically make a table of contents.  This is done by
%% uncommenting the following:
%%

%\tableofcontents

%%%%%%%%%%%%%%%%%%%%%%%%%%%%%%%%%%%%%%%%%%%%%%%%%%%%%%%%%%%%%%%%%%%%%%
\section{Introduction}
%%%%%%%%%%%%%%%%%%%%%%%%%%%%%%%%%%%%%%%%%%%%%%%%%%%%%%%%%%%%%%%%%%%%%%

Yau's Theorem \cite{Yau78} \cite{joyce2007riemannian} \cite{szekelyhidi2014introduction} for the Calabi conjecture \cite{Calabi57}, proven forty years ago, occupies a central place in the theory of K\"ahler manifolds and had wide-ranging applications in geometry and mathematical physics \cite{FuYau08,Yau77}. \par

%\TOC{Let $M$ be a closed K\"ahler manifold of complex dimension $n$ with the K\"ahler form $\omega$ and a compatible complex structure $J$. Yau proved that given a volume form $e^{f}\omega^n$ on $M$ there exists a unique smooth function $f$ satisfying
%     \begin{equation}
%        \begin{split}
%            (\omega + \sqrt{-1}\partial_J \bar{\partial}_J \varphi)^n = e^{f}\omega^n \\
%            \omega + \sqrt{-1}\partial_J \bar{\partial}_J \varphi > 0, \quad \sup_M \varphi= 0
%        \end{split}
%    \end{equation}
%if and only if $\varphi$ satisfies the condition: $$\int_M f \omega^n = 0.$$}{Replace Yau's Theorem}

The theorem is equivalent to finding a K\"ahler metric within a given K\"ahler class that has a prescribed Ricci form. In other words, this amounts to solving the complex Monge-Amp\`ere equation on K\"ahler manifolds:
\begin{equation}\label{eq:1.1}
            (\omega + \sqrt{-1}\partial_J \bar{\partial}_J \varphi)^n = e^{f}\omega^n
\end{equation}
for a smooth real function $\varphi$ satisfying $\omega + \sqrt{-1}\partial_J \bar{\partial}_J \varphi > 0,$ and $\sup_M \varphi= 0$, where $n$ is the complex dimension of $M$, and $f$ is any smooth real function with 
\begin{equation*}
    \int_M e^f \omega^n = \int_M \omega^n.
\end{equation*}
\par

There has been significant interest in extending Yau's Theorem to non-K\"ahler settings. One extension of Yau's Theorem, initiated by Cherrier \cite{Cherrier87} in the 1980s, involves removing the closedness condition $d\omega = 0$. See also Tosatti-Weinkove \cite{TosattiW10}, and Fu-Yau \cite{FuYau08}. The Monge-Amp\`ere equation on almost Hermitian manifolds was studied by Chu-Tosatti-Weinkove \cite{Chu20,ChuTW19}. A different extension on symplectic manifolds was developed by Weinkove \cite{Weinkove07} and Tosatti-Weinkove-Yau \cite{TosattiWY08}, who studied the Calabi-Yau equation for 1-forms. Delan\"oe \cite{Delanoe96} and Wang-Zhu \cite{WangZ10} considered a Gromov-type Calabi-Yau equation. \cref{eq:1.1} is a special case of general fully non-linear elliptic equations (cf. \cite{HarveyLawson15} \cite{Szekelyhidi18} \cite{guan2014second} \cite{chu2023fully} \cite{TWWY15} \cite{huang2023fully} \cite{tosatti2019complex} \cite{Chu20}).\par

This paper focuses on a generalized Monge-Amp\`ere equation on almost K\"ahler surfaces and establishes a uniqueness theorem for its solutions. Our main result is as follows:
\begin{thm}
\label{thm:1}
    Suppose that $(M,\omega,J,g)$ is a closed almost K\"ahler surface, then there exists at most one solution, $\varphi \in C^{\infty}(M,J)_0$, of the generalized Monge-Amp\`ere equation 
    \begin{equation} \label{eq:1.2}
        (\omega + \mathcal{D}_J^+(\varphi))^2 = e^{f}\omega^2,
    \end{equation}
     for $\varphi$ satisfying $\omega + \mc{D}_J^+(\varphi) > 0$, where $f$ is any smooth real function with
     \begin{equation*}
         \int_M \omega^2 = \int_M e^{f}\omega^2.
     \end{equation*}
     Moreover, if $\varphi_0$ (defined by $\varphi$) is in $\mathcal{H}'(\omega,J)$, then there is an $L^1$ $a$ $priori$ bound of $\varphi_0$ depending only on $\omega,J,$ and $f$.
     %and a $C^{\infty}$ $a\ priori$ bound of $\varphi_0$ depending only on $\omega,J,$ $f$, and $\sup_M\tr_gg_1$. Here $g_1$ is the Riemannian metric corresponding to $\omega_1=\omega + \mc{D}_J^+(\varphi)$ and 
     %\[
     %\mathcal{H}'(\omega,J):=\{\phi\in C^\infty|\omega+(dJd\phi)^{1,1}>0\}.
     %\]
\end{thm}

%Lejmi \cite{Lejmi10E,Lejmi10S} defined a self-adjoint strongly elliptic linear operator
%\begin{equation*}
%    P = d_J^-d^*: \Omega_J^- \to \Omega_J^-.
%\end{equation*}
%with $\Omega_J^- = \ker P \bigoplus d_J^- (\Omega_{\R}^1)$, where $\ker P = \mathcal{H}_J^-$ is the space of $J$-anti-invariant self-%dual harmonic 2-forms. We write $h_J^-$ to denote the dimension of $\mathcal{H}^-_J$, then $0 \leq h_J^- \leq b^+$ (see %\cite{TanWZZ15}). Thus we can define
%\begin{equation*}
%    \mathcal{D}_J^+: C^{\infty}(M) \to \Omega_J^+(M) = \Omega_J^{1,1}(M)
%\end{equation*}
%as follows (see Tan-Wang-Zhou-Zhu \cite{TanWZZ22}): \\

%For any $f \in C^{\infty}(M)$, define 
%\begin{equation*}
%    \mathcal{W}_J(f) = Jdf + d^*(\eta_f + \bar{\eta}_f),
%\end{equation*}
%satisfying (see Weinkove \cite{Weinkove07})
%\begin{equation*}
%    d_J^-Jdf + dd^*(\eta_f + \bar{\eta}_f) = 0,
%\end{equation*}
%where $\eta_f \in \Omega_J^{0,2}(M)$ (due to Lejmi's operator $P$). \\
%
%Define
%\begin{equation*}
%    \mathcal{D}_J^+(f) = d\mathcal{W}_J(f) \in \Omega_J^+(M)
%\end{equation*}

We explain some of the notations appearing in the main theorem. The operator $\mathcal{D}_J^+$, introduced by Tan-Wang-Zhou-Zhu \cite{TanWZZ22}, generalizes $\del_J\bar{\del}_J$. In fact, when $J$ is integrable, $\mathcal{D}_J^+$ reduces to $2\sqrt{-1}\del_J\bar{\del}_J$. Therefore, it may be regarded as a generalization of $\del_J\bar{\del}_J$ to the almost complex setting. Using $\mathcal{D}_J^+$, Tan-Wang-Zhou-Zhu \cite{TanWZZ22} resolved Donaldson's tameness question. Moreover, Wang-Wang-Zhu \cite{WangWZ23} subsequently derived a Nakai-Moishezon criterion for almost complex 4-manifolds. Recall that for a closed almost K\"ahler surface $(M,\omega,J,g)$, the inequality $0 \leq h_J^- \leq b^+ - 1$ holds (\cite{TanWZZ15, TanWZZ22}). %Observe that the intersection of $H_J^+$ and $H_g^+$ is spanned by ${\omega, f_i \omega + d_J^-(\nu_i + \bar{\nu}_i)}$, where $\nu_i \in \Omega_J^{0,1}(M)$ and
%\begin{equation*}
 %   \int_M f_i \omega^2 = 0
%\end{equation*}
%for $1 \leq i \leq b^+ - h_J^- - 1$. Note that the kernel of $\mathcal{W}_J$ is spanned by $\Set{1, f_i,\ 1 \leq i \leq b^+- h_J^- -1}$. Let
%$    C^{\infty}(M,J)_0 := C^{\infty}(M)_0 \setminus \mathrm{Span}\ \{f_i,\ 1 \leq i \leq b^+- h_J^- -1\},$
%where 
%\begin{equation*}
%    C^{\infty}(M)_0 := \Set{f \in C^{\infty}(M) \given \int_M f\omega^2 = 0}.
%\end{equation*}
%Thus, $C^{\infty}(M,J)_0 \subset C^{\infty}(M)_0$; they are equal if $h_J^- = b^+-1$. 
The function $\varphi_0$ is defined as the solution of 
\[
-\frac12\Delta_g\varphi_0=\frac{\omega\wedge\mathcal{D}_J^+(\varphi)}{\omega^2}.
\]
Similarly, $\varphi_1$ is defined as the solution of 
\[
-\frac12\Delta_{g_1}\varphi_1=\frac{\omega_1\wedge\mc{D}_J^+(\varphi)}{\omega^2_1}.
\]
As a consequence of this, $\mathcal{D}_J^+(\varphi)=dJd\varphi_0+da(\varphi_0)=dJd\varphi_1+da(\varphi_1)$, where $a(\varphi_0)$ and $a(\varphi_1)$ are solutions of the elliptic system (\cref{equ:ellipsys}). This identity is proved after \cref{lem:1}.

\begin{rem}
    \begin{itemize}
        \item[1)] The upper bound of $\tr_gg_1$ is controlled by Tian's $\alpha$-integral \cite{Tian87}, 
        \[
        I_\alpha(\varphi_1):=\int_Me^{-\alpha\varphi_1}\omega_1^2,\ \alpha>0.
        \]
        Details can be found in Tosatti-Weinkove-Yau \cite[Theorem~3.1,~4.1]{TosattiWY08}. Here \cref{thm:1} is analogous to \cite[Theorem~1.3]{TosattiWY08}.
        \item[2)] Tosatti-Weinkove \cite[Theorem~1.2]{TosattiW18} shows that on a Kodaira-Thurston manifold $M$, one can obtain a unique solution to \cref{eq:1.2} depending only on $M,\omega,J,$ and $f$.
    \end{itemize}
\end{rem}

\noindent Motivated by the results of Tan-Wang-Zhou-Zhu \cite{TanWZZ22}, it is natural to ask the following question:

\begin{Question}
    1$)$ Suppose that $(M,\omega, J, g)$ is a closed almost K\"ahler surface. If
    \begin{equation*}
        \int_M e^{f} \omega^2 = \int_M \omega^2,\quad b^+-1=h_J^-,
    \end{equation*}
    is there a function $\varphi \in C^{\infty}(M)$, unique up to the addition of a constant, such that ${\omega}_1 = \omega + \mathcal{D}_J^+(\varphi)$ solving the following equation
    \begin{equation*}
    {\omega}_1^2 = e^{f} \omega^2,
    \end{equation*}
    and there is $a\ priori$ $C^{\infty}$-bound on $\varphi$ depending only on $\omega,J,f$, and $M$?\\

    2$)$ Let $(M,\omega,J)$ be a compact almost K\"ahler manifold. Are there two positive constants $\alpha,C$ depending only on $(M,\omega,J)$ such that 
    \[
    \int_Me^{-\alpha\varphi}d\mathrm{vol}<C,
    \]
    for each $C^2$ function $\varphi$ with $\sup\varphi=0$ or $\int_M\varphi d\mathrm{vol}=0$.
\end{Question}

%\begin{rem}
%    \begin{enumerate}
%        \item The authors are thankful for the referee for pointing out that proving a uniform $L^1$ bound for $\varphi$ is already sufficient to settle Donaldson's conjecture \cite{TosattiW18}.
%        \item It is natural to consider a generalized $\del_J \overline{\del}_J$ operator
%        \begin{equation*}
%            \mc{D}_J^+: C^{\infty}(M^{2n}) \to \Omega_J^+(M^{2n})
%        \end{equation*}
%        on an almost K\"ahler manifolds $(M^{2n},\omega,J)$ of complex dimension $n \geq 3$, and study the generalized Monge-Amp\`ere equation:
%        \begin{equation}
%            (\omega + \mc{D}_J^+(\varphi))^n = e^f \omega^n,
%        \end{equation}
%        where $\varphi,f \in C^{\infty}(M^{2n})$ satisfying
%        \begin{equation}
%            \int_{M^{2n}} \omega^n = \int_{M^{2n}} e^f \omega^n
%        \end{equation}
%    \end{enumerate}
%    
%\end{rem}

Section 2 introduces the notations for almost K\"ahler manifolds and defines the operator $\mathcal{D}_J^+$. Additionally, a local theory for the generalized Calabi-Yau equation is presented.
In Section 3, the uniqueness part of the main theorem is proved. Finally, Section 4 provides an $L^1$ $a\ priori$ estimate for the function $\varphi_0$.

%%%%%%%%%%%%%%%%%%%%%%%%%%%%%%%%%%%%%%%%%%%%%%%%%%%%%%%%%%%%%%%%%%%%%%
\section{Preliminaries}
%%%%%%%%%%%%%%%%%%%%%%%%%%%%%%%%%%%%%%%%%%%%%%%%%%%%%%%%%%%%%%%%%%%%%%

Let $(M, J)$ be an almost complex manifold of dimension $2n$. A Riemannian metric $g$ on $M$ is said to be compatible with the almost complex structure $J$ if
\begin{equation*}
    g(JX, JY) = g(X, Y),
\end{equation*}
for all tangent vectors $X,Y \in TM$. In this case, $(M, J, g)$ is called an almost-Hermitian manifold. \par

The almost complex structure $J$ induces a decomposition of the complexified tangent space $T^{\C}M$:
\begin{equation*}
    T^{\C}M = T^{1,0}M \oplus T^{0,1}M,
\end{equation*}
where $T^{1,0}M$ and $T^{0,1}M$ are the eigenspaces of $J$ corresponding to the eigenvalues $\sqrt{-1}$ and $-\sqrt{-1}$, respectively. \par

A local unitary frame ${e_1, \ldots, e_n}$ can be chosen for $T^{1,0}M$, with the dual coframe denoted by ${\theta^1, \ldots, \theta^n}$. Using this coframe, the metric $g$ can be expressed as
\begin{equation*}
    g = \theta^i \otimes \overline{\theta^i} + \overline{\theta^i} \otimes \theta^i.
\end{equation*}
The fundamental form $\omega$ is defined by $\omega(\cdot, \cdot) = g(J\cdot,\cdot)$ and can be written as
\begin{equation*}
    \omega = \sqrt{-1}\theta^i \wedge \overline{\theta^i}.
\end{equation*}
The manifold $(M, \omega, J, g)$ is called almost K\"ahler if $d\omega = 0$. \par

The almost complex structure $J$ also acts as an involution on the bundle of real two-forms via
\begin{equation*}
    \mc{J}:\alpha(\cdot,\cdot) \mapsto \alpha(J\cdot,J\cdot).
\end{equation*}
This action induces a splitting of $\Lambda^2$ into $J$-invariant and $J$-anti-invariant two-forms:
\begin{equation*}
    \Lambda^2 = \Lambda_J^+ \oplus \Lambda_J^-.
\end{equation*}
Let $\Omega_J^+$ and $\Omega_J^-$ denote the spaces of the $J$-invariant and $J$-anti-invariant forms, respectively. We use $\mathcal{Z}$ to denote the space of closed 2-forms and $\mathcal{Z}_J^{\pm} := \mathcal{Z} \cap \Omega_J^{\pm}$ for the corresponding subspaces. \par

The following operators are defined as:
\begin{align*}
    d_J^+ := P_J^+d: \Omega_{\R}^1 \to \Omega_J^+, \\
    d_J^- := P_J^-d: \Omega_{\R}^1 \to \Omega_J^-,
\end{align*}
where $P_J^{\pm} = \frac{1}{2}(1\pm J)$ are algebraic projections on $\Omega_{\R}^2(M)$. \\

\begin{prop}
Let $(M,J,g)$ be a closed Hermitian 4-manifold, then
\begin{equation*}
    d_J^+: \Lambda_{\R}^1 \otimes L_1^2(M) \to \Lambda_J^{1,1} \otimes L^2(M)
\end{equation*}
has closed range.
\end{prop}

Li and Zhang \cite{LiZhang09} introduced the $J$-invariant and $J$-anti-invariant cohomology subgroups $H_J^{\pm}$ of $H^2(M;\R)$ as follows:
\begin{defn}
The $J$-invariant, respectively, $J$-anti-invariant cohomology subgroups $H_J^{\pm}$
are defined by
$$
    H_J^{\pm} := \Set{\mathfrak{a} \in H^2(M,\R) \given \exists \alpha \in \mathcal{Z}_J^{\pm} \text{ such that } [\alpha] = \mathfrak{a}}.
$$
An almost complex structure $J$ is said to be $C^{\infty}$-pure if $H_J^+ \cap H_J^- = \{ 0 \}$, respectively $C^{\infty}$-full if 
$H_J^+ + H_J^- = H^2(M;\R)$.
\end{defn}

In the case of (real) dimension 4, this gives a decomposition of $H^2(M)$:
\begin{prop}[\cite{DLZ10}]
If $M$ is a closed almost complex 4-manifold, then any almost complex structure $J$ on $M$ is $C^{\infty}$-pure and full, i.e.,
$$
    H^2(M;\R) = H_J^+ \oplus H_J^-.
$$
Let $h^+_J$ and $h_J^-$ denote the dimensions of $H^+_J$ and $H_J^-$, respectively. Then $b^2 = h^+_J + h_J^-$ , where $b^2$ is the second Betti number.
\end{prop}

It is well-known that the self-dual and anti-self-dual decomposition of 2-forms is induced by the Hodge operator $\hodge_g$ of a Riemannian metric $g$ on a 4-dimensional manifold $M$:
\begin{equation*}
    \Lambda^2 = \Lambda_g^+ \oplus \Lambda_g^-.
\end{equation*}
Let $\Omega^\pm_g$ denote the spaces of smooth sections of $\Lambda^{\pm}_g$. The Hodge-de Rham Laplacian,
$$
    \Delta_g=dd^*+d^*d:\Omega^2(M)\rightarrow\Omega^2(M),
$$
where $d^*=-\hodge_g d \hodge_g$ is the codifferential operator with respect to the metric $g$, commutes with Hodge star operator $\hodge_g$. Consequently, the decomposition also holds for the space $\mathcal{H}_g$ of harmonic 2-forms. By Hodge theory, this induces a cohomology decomposition determined by the metric $g$:
$$
  \mathcal{H}_g=\mathcal{H}_g^+\oplus\mathcal{H}_g^-.
$$

We can further define the operators
\begin{equation*}
  d^\pm_g := P_g^{\pm}d:\Omega^1_\R \to \Omega^\pm_g,
\end{equation*}
where $d$ is the exterior derivative $d : \Omega^1_\R \to \Omega^2_\R$, and $P^\pm_g := \frac{1}{2}(1 \pm \hodge_g)$ are the algebraic projections. The following Hodge decompositions hold:
\begin{equation*}
   \Omega^+_g = \mathcal{H}^+_g \oplus d^+_g(\Omega_{\R}^1), \quad \Omega^-_g=\mathcal{H}^-_g \oplus d^-_g(\Omega_{\R}^1).
\end{equation*}
These decompositions are related by \cite{TanWZZ22}
\begin{equation}\label{equ:decom}
    \begin{aligned}
    \Lambda_J^+ &= \R \omega \oplus \Lambda_g^-,\\
    \Lambda_g^+ &= \R \omega \oplus \Lambda_J^-.
    \end{aligned}
\end{equation}
In particular, any $J$-anti-invariant 2-form in 4 dimensions is self-dual. Therefore, any closed $J$-anti-invariant 2-form is harmonic and self-dual. This identifies the space $H_J^-$ with $\mathcal{Z}_J^-$ and, further, with the set $\mathcal{H}_g^{+,\omega^\perp}$ of harmonic self-dual forms that are pointwise orthogonal to $\omega$. \\

Lejmi \cite{Lejmi10E,Lejmi10S} first recognized $\mathcal{Z}^-_J$ as the kernel of an elliptic operator on $\Omega^-_J$:
\begin{prop}[\cite{Lejmi10E}]
\label{prop:3}
Let $(M,\omega,J,g)$ be a closed almost Hermitian 4-manifold. Define the following operator
\begin{align*}
    P: \Omega^-_J &\to \Omega^-_J\\
    \psi&\mapsto P^-_J(dd^*\psi).
\end{align*}
Then $P$ is a self-adjoint, strongly elliptic linear operator with a kernel consisting of $g$-self-dual harmonic, $J$-anti-invariant 2-forms. Hence,
\begin{equation*}
    \Omega^-_J=\ker P \oplus d^-_J\Omega^1_\R.
\end{equation*}
\end{prop}

By using the operator $P$ defined in \cref{prop:3}, Tan-Wang-Zhou-Zhu \cite{TanWZZ22} introduced the $\mc{D}^+_J$ operator:
\begin{defn}
 Let $(M,\omega,J,g)$ be a closed almost Hermitian 4-manifold. Denote by
$$
L_2^2(M)_0:=\{f\in L_2^2(M)|\int_Mfd\mu_g=0\}.
$$
Define
\begin{align*}
    \mathcal{W}_J: L^2_2(M)_0 &\longrightarrow \Lambda^1_\mathbb{R}\otimes L^2_1(M), \\
                    f &\longmapsto Jdf+d^*(\eta^1_f+\overline{\eta}^1_f),
\end{align*}
where $\eta^1_f \in \Lambda^{0,2}_J \otimes L^2_2(M)$ satisfies
$$
    d^-_J\mathcal{W}_J(f)=0.
$$
Define
\begin{align*}
    \mathcal{D}^+_J: L^2_2(M)_0 &\longrightarrow \Lambda^{1,1}_J\otimes L^2(M), \\
    f &\longmapsto d\mathcal{W}_J(f).
\end{align*}
In the case of $(M,\omega,J,g)$ being an almost K\"ahler surface, such a function $f$ is called the almost K\"ahler potential with respect to the almost K\"ahler metric $g$.
\end{defn} 

%\begin{rem}
%\label{rem:1}
%    If $J$ is integrable, that is $N_J = 0$, then $\mathcal{D}^+_J(f) = 2\sqrt{-1}\del_J\bar{\del}_Jf$. So
%    we can regard $\mathcal{D}^+_J$ as a generalized $\partial\bar{\partial}$-operator.
%\end{rem}

The operator $\mc{D}_J^+$ has closed range as well (cf. \cite{TanWZZ22}):
\begin{prop}
Suppose $(M,\omega,J,g)$ is a closed almost K\"ahler surface. Then $\mc{D}_J^+: L_2^2(M) \to \Lambda_J^+ \otimes L^2(M)$ has closed range.
\end{prop}

The remainder of this section is devoted to a local theory of a generalized Calabi-Yau equation suggested by Gromov \cite{Delanoe96,WangZ10}. \par

Observe that the generalized Monge-Amp\`ere equation \labelcref{eq:1.2} is equivalent to the following Calabi-Yau equation:
\begin{equation}  \label{eq:3.3}
    (\omega + du)^2 = e^f \omega^2 
\end{equation}
for $u \in \Omega_\R^1(M)$, $d_J^- u = 0$, by letting $u = \mc{W}_J(\varphi)$.

\begin{defn}\label{def:2.7}
    Suppose $(M,\omega,J,g)$ is a closed almost K\"ahler surface. The sets $A,B,A_+$, and $B_+$ are defined as follows:
    \begin{alignat*}{2}
        &A & &:=\ \Set{u \in \Omega_\R^1(M) \given d_J^- u = 0, \quad d^* u = 0}, \\
        &B & &:=\ \Set{\varphi \in C^{\infty}(M) \given \int_M \varphi\omega^2 = \int_M \omega^2}, \\
        &A_+ & &:=\ \Set{u \in A \given \omega + du > 0}, \\[0.5em]
        &B_+ & &:=\ \Set{f \in B \given f > 0 \text{ on } M}.
    \end{alignat*}
\end{defn}

Let $\omega(\phi) = \omega + d\phi$. Since 
$$
\int_M(\omega(\phi))^2=\int_M\omega^2
$$
with $\phi\in A$, the operator $\mc{F}$, defined by
\begin{equation*}
    \mc{F}(\phi)\omega^2 = (\omega(\phi))^2
\end{equation*}
sends $A$ into $B$. \par

By a direct calculation, for any $u \in A_+$, the tangent space $T_u A_+$ at $u$ is $A$. Given any $\phi \in A$, we define
\begin{equation}
    L(u)(\phi) = \diff*{\mc{F}(u+t\phi)}{t}[t=0].
\end{equation}
According to \cite{Delanoe96,WangZ10}, $L(u)$ is a linear elliptic system on $A$. Hence, we get the following lemma (cf. \cite[Proposition~1]{Delanoe96} ,\cite[Lemma~2.5]{WangZ10}):
\begin{lem}
    Suppose that $(M,\omega,J,g)$ is a closed almost K\"ahler surface. Then the restricted operator
    \begin{equation*}
        \mc{F}|_{A_+}: A_+ \to B_+
    \end{equation*}
    is of elliptic type on $A_+$.
\end{lem}
Obviously, $A_+ \subset A$ is an open convex set. As done in \cite{WangZ10},
\begin{equation*}
    \mc{F}|_{A_+}: A_+ \to B_+
\end{equation*}
is injective. \\

In summary, by nonlinear analysis \cite{Aubin98}, the following result (cf. Delano\"e \cite[Theorem~2]{Delanoe97} or Wang-Zhu \cite[Proposition~2.6]{WangZ10}) is true: 
\begin{prop}
\label{prop:4}
    The restricted operator
    \begin{equation*}
        \mc{F}|_{A_+}: A_+ \to \mc{F}(A_+) \subset B_+
    \end{equation*}
    is a diffeomorphic map.
\end{prop}

\begin{rem}
    In fact, $\mc{F}(A_+) = B_+$ is equivalent to the existence theorem for the generalized Monge-Amp\`ere \labelcref{eq:1.2} on closed almost K\"ahler surfaces (cf. Delano\"e \cite{Delanoe96}, Wang-Zhu \cite{WangZ10}).
\end{rem}

%%%%%%%%%%%%%%%%%%%%%%%%%%%%%%%%%%%%%%%%%%%%%%%%%%%%%%%%%%%%%%%%%%%%%%
\section{Uniqueness Theorem for the generalized Monge-Amp\`ere equation}
%%%%%%%%%%%%%%%%%%%%%%%%%%%%%%%%%%%%%%%%%%%%%%%%%%%%%%%%%%%%%%%%%%%%%%

This section aims at demonstrating the uniqueness part of the main theorem. Throughout this section, we assume that $(M, \omega, J, g)$ is a closed almost K\"ahler surface. If $\omega_1 = \omega + \mc{D}_J^+(\varphi) > 0$ for some $\varphi$, the metric $g_1(\cdot, \cdot)$ is defined by $g_1(\cdot, \cdot) = \omega_1(\cdot, J \cdot)$. Let $\hodge_g$ and $\hodge_{g_1}$ denote the Hodge star operators corresponding to the metrics $g$ and $g_1$, respectively. \par

Suppose $\varphi_0\in C^\infty(M)_0$ satisfies the following equation
\begin{equation*}
    \omega_1 \wedge \mc{D}_J^+(\varphi_0) = (\omega + \mc{D}_J^+(\varphi)) \wedge \mc{D}_J^+(\varphi_0) = 0.
\end{equation*}
This implies that $P_{g_1}^+ \mc{D}_J^+(\varphi_0) = 0$, that is, $\mc{D}_J^+(\varphi_0)$ is $g_1$-anti-self-dual, since  
$$
    \Lambda_J^+ = \R \omega_1 \oplus d_{g_1}^-(\Omega^1(M)).
$$
Thus
\begin{equation*}
    \mc{D}_J^+(\varphi_0) = d\mc{W}_J(\varphi_0) = d_{g_1}^-\mc{W}_J(\varphi_0).
\end{equation*}
But
\begin{align*}
    0 = \int_M \mc{D}_J^+(\varphi_0) \wedge \mc{D}_J^+(\varphi_0) &= \int_M d_{g_1}^-\mc{W}_J(\varphi_0) \wedge d_{g_1}^-\mc{W}_J(\varphi_0) \\
      &= - \norm{d_{g_1}^-\mc{W}_J(\varphi_0)}_{g_1}^2,
\end{align*}
hence
\begin{equation}
    \mc{D}_J^+(\varphi_0) = d\mc{W}_J(\varphi_0) = 0.
\end{equation}
Since
\begin{equation}
    d^* \mc{W}_J(\varphi_0) = 0.
\end{equation}
we have $\mc{W}_J(\varphi_0) = 0$. Hence $\varphi_0$ is a constant. \par

We now suppose that there are two solutions $\varphi_1$ and $\varphi_2$, i.e.,
\begin{equation*}
     (\omega + \mathcal{D}_J^+(\varphi_1))^2 = (\omega + \mathcal{D}_J^+(\varphi_2))^2 = e^{f}\omega^2 .
\end{equation*}
For each $t \in [0,1]$, set $\varphi_t = t \varphi_1 + (1-t) \varphi_2$, then
\begin{equation*}
    \int_0^1 \diff*{(\omega + \mathcal{D}_J^+(\varphi_t))^2}{t} \dl3t = 0.
\end{equation*}
A direct calculation of $\diff*{(\omega + \mathcal{D}_J^+(\varphi_t))^2}{t}$ shows
\begin{equation*}
    0=\int_0^1 \diff*{(\omega + \mathcal{D}_J^+(\varphi_t))^2}{t} \dl3t=(2\omega + \mathcal{D}_J^+(\varphi_1+\varphi_2)) \wedge \mc{D}_J^+(\varphi_1 - \varphi_2).
\end{equation*}
This implies that $\varphi_1 - \varphi_2$ is a constant. Thus, as in the K\"ahler case \cite{Calabi57}, we obtain a uniqueness theorem for the generalized Monge-Amp\`ere equation:
\begin{thm}\label{thm:3.1}
    The generalized Monge-Amp\`ere equation \labelcref{eq:1.2} on a closed almost K\"ahler surface has at most one solution up to the addition of a constant.
\end{thm}

%%%%%%%%%%%%%%%%%%%%%%%%%%%%%%%%%%%%%%%%%%%%%%%%%%%%%%%%%%%%%%%%%%%%%%
\section{a priori estimates for the solution of the generalized Monge-Amp\`ere equation}
%%%%%%%%%%%%%%%%%%%%%%%%%%%%%%%%%%%%%%%%%%%%%%%%%%%%%%%%%%%%%%%%%%%%%%

In this section, we establish an $L^1$ a  priori estimate for $\varphi_0$. Consider a closed almost K\"ahler surface $(M, \omega, J, g)$. Recall that the Calabi-Yau equation \labelcref{eq:3.3} is equivalent to the generalized Monge-Amp\`ere equation
\begin{equation*}
      (\omega + d\mathcal{W}_J(\varphi))^2 = e^{f}\omega^2, 
\end{equation*}
where 
\begin{equation*}
    \int_M \omega^2 = \int_M e^{f}\omega^2,
\end{equation*}
$d\mathcal{W}_J(\varphi) = \mathcal{D}_J^+(\varphi)$ and $d^{\hodge} \mathcal{W}_J(\varphi) = 0$.

Assume that
$$
\omega_1^2=e^f\omega^2.
$$ 
We now define a function $\varphi_0 \in C^{\infty}(M)$ by
\begin{equation}\label{equ:phi0}
    -\frac{1}{2}\Delta_g \varphi_0 = \frac{\omega \wedge (\omega_1 - \omega)}{\omega^2},
\end{equation}
where $\Delta_g$ denotes the Laplacian associated with the Levi-Civita connection of the almost K\"ahler metric $g$. Note that $\Delta_g=\Delta_g^c$, where $\Delta_g^c$ is the complex Laplacian with respect to $g$ (cf. \cite[Lemma~2.6]{TosattiWY08}). In general, $\varphi_0\neq\varphi$. The existence of $\varphi_0$ follows from elementary Hodge theory and it is uniquely determined up to the addition of a constant. Since 
\begin{equation}\label{eq:4.22}
    - \omega \wedge dJd\varphi = \frac{1}{2}\Delta_g \varphi \omega^2
\end{equation}
for any almost K\"ahler form $\omega$ associated with $g$ and any function $\varphi$, it follows from Lejmi's Theorem (\cref{prop:3}) that there exists $\sigma(\varphi) \in \Omega_J^-$ satisfying the elliptic system:
\begin{equation} \label{eq:4.21}
        d_J^- J d\varphi + d_J^- d^* \sigma(\varphi) = 0.
\end{equation}
Hence,
\begin{equation*}
    \begin{aligned}
        \omega_1 - \omega = \mc{D}_J^+(\varphi) &= dJ d\varphi + dd^* \sigma(\varphi) \\
            & = dd^* (\varphi \omega) + dd^* \sigma(\varphi).
    \end{aligned}
\end{equation*}

Therefore $\mc{W}_J(\varphi)$ can be rewritten as $d^* (\varphi \omega) + d^* \sigma(\varphi)$. Then
\begin{equation*}
    \begin{dcases}
        d \mc{W}_J(\varphi) = \omega_1 - \omega, \\
        d^* \mc{W}_J(\varphi) = 0.
    \end{dcases}
\end{equation*}

Let $\omega_1 = \omega + \mc{D}_J^+(\varphi)$, where both $\omega_1$ and $\omega$ are symplectic forms with $[\omega_1] = [\omega]$. For any $p \in M$, by the Darboux Theorem, we may assume, without loss of generality, that on a Darboux coordinate neighborhood $U_p$:
\begin{equation} \label{eq:4.1}
\begin{aligned}
    \omega(p) &= \sqrt{-1}(\theta^1 \wedge \overline{\theta^1} + \theta^2 \wedge \overline{\theta^2}), \\
    g(p) &= 2(|\theta^1|^2 + |\theta^2|^2),
\end{aligned}
\quad
\begin{aligned}
    \omega_1(p) &= \sqrt{-1}(a_1 \theta^1 \wedge \overline{\theta^1} + a_2 \theta^2 \wedge \overline{\theta^2}), \\
    g_1(p) &= 2(a_1 |\theta^1|^2 + a_2 |\theta^2|^2),
\end{aligned}
\end{equation}
where $0 < a_1 \le a_2$ (using simultaneous diagonalization).

\begin{lem}
\label{lem:1}
    For any point $p$ in an almost K\"ahler surface $M$, using the coordinates in \cref{eq:4.1}, and supposing that $\varphi$ is a solution of \cref{eq:1.2}, we have
    \begin{equation*}
        e^{f(p)} = a_1 a_2 \leq \abs{g_1(p)}_g^2 , \quad \abs{d \mc{W}_J(\varphi)(p)}_g^2 = 2 [(a_1-1)^2 + (a_2-1)^2],
    \end{equation*}
    and
    \begin{equation*}
    {\Delta_g \varphi_0}(p) = 2 - (a_1 + a_2) \leq 2(1 - e^{\frac{f(p)}{2}}) < 2.
    \end{equation*}
\end{lem}

\begin{proof}
    Since
    \begin{equation*}
    \begin{aligned}
        \dl \vol_{g_1}|_p = \frac{\omega_1^2(p)}{2!} &= -a_1 a_2 \theta^1 \wedge \overline{\theta^1} \wedge \theta^2 \wedge \overline{\theta^2} \\
            &= -e^{f(p)} \theta^1 \wedge \overline{\theta^1} \wedge \theta^2 \wedge \overline{\theta^2}\ (\text{by } \labelcref{eq:1.2}),
%            \dl z_1 \wedge \dl \bar{z}_1 \wedge \dl z_2 \wedge \dl \bar{z}_2\
    \end{aligned}
    \end{equation*}
    then {$e^{f(p)}=a_1 a_2\le2(a_1^2+a_2^2)=\abs{g_1(p)}_g^2$}. The remaining identities follow from direct computation using \labelcref{eq:4.1}.
\end{proof}

%\begin{cor}
%     For any $\varphi$, let $\varphi_0$ and $da(\varphi_0)$ be the solutions of \cref{equ:phi0,equ:aphi0}. If $||\tr_gg_1||_{C^0}$ is bounded from above, then there is a uniform constant depending only on $M,\omega,J$, $f$, and $||\tr_gg_1||_{C^0}$ such that
%    \[
%    ||\Delta_g\varphi_0||_{C^0(M,g)}\le C.
%    \]
%\end{cor}

Consider a family of symplectic forms on the almost K\"ahler surface $(M,\omega,J,g)$
\begin{equation*}
    \omega_s = (1-s)\omega + s\omega_1,\ s \in [0,1].
\end{equation*}
Then, $\omega_0 = \omega$, $\omega_{\frac{1}{2}} = \frac{1}{2}(\omega+\omega_1)$. Moreover, we have
\begin{equation}
\label{eq:4.2}
    -2\omega_{\frac{1}{2}} < \omega_1 - \omega < 2\omega_{\frac{1}{2}}.
\end{equation}
Let $g_s(\cdot,\cdot) = \omega_s(\cdot,J\cdot)$ for $0\le s\le1$ and $d^{\hodge_s} = - \hodge_{g_s} d \hodge_{g_s}$. Define the almost K\"ahler potentials $\varphi_s$, unique up to the addition of a constant \cite{Weinkove07}, by
\begin{equation} \label{eq:4.3}
    -\frac{1}{2} \Delta_{g_s} \varphi_s = \frac{\omega_s \wedge (\omega_1 -\omega)}{\omega_s^2}
\end{equation}
Note that in general, $\varphi_s \neq \varphi$ for $0\le s\le1$. According to \cref{prop:3}, since $\omega_1 - \omega \in \Omega_J^+ \cap d(\Omega^1)$, the self-dual equation implies that (cf. \cite[Section~2]{Weinkove07})
\begin{equation} \label{eq:4.4}
    \omega_1 - \omega =\mc{D}_J^+(\varphi)= \mc{D}_{J,s}^+(\varphi_s) := dJd\varphi_s + da(\varphi_s),\ s \in [0,1],
\end{equation}
where $a(\varphi_s)\in\Omega^1_\R$. The pair $(\varphi_s,a(\varphi_s))$ satisfies the elliptic system
\begin{equation}\label{equ:ellipsys}
\begin{dcases}
    d^{*_s}a(\varphi_s)=0,\\
    d_J^-Jd\varphi_s + d_{J}^-a(\varphi_s) = 0,\\
    \omega_s\wedge da(\varphi_s)=0,
\end{dcases}
\end{equation}
where $d_J^-a(\varphi_s)$ and $d_J^+a(\varphi_s)$ are the $g_s$-self-dual and $g_s$-anti-self-dual parts (\cref{equ:decom}), respectively.

Since the solution $\varphi_0$ of the Monge-Amp\`ere equation is an almost K\"ahler potential and is unique up to addition of a constant. In the rest of this section, we only consider almost K\"ahler potentials normalized by $\sup_{M}\varphi_0=0$. We give an $L^1$-estimate for $\varphi_0$ first:

\begin{prop}[cf. {\cite[Proposition~2.3]{ChuTW19}}]\label{lem:4.3}
    Let $(M,\omega,g,J)$ be a compact almost K\"ahler surface,
    then there is a constant $C>0$ depending only on $M,\omega,g,J$, and $f$ such that
    \[
    ||\varphi_0||_{L^1}=\int_M(-\varphi_0)\omega^2\le C,
    \]
    for all smooth functions $\varphi_0$ on $M$ satisfying
    \[
    \omega+\mathcal{D}_J^+(\varphi)=\omega+dJd\varphi_0+da(\varphi_0)>0,\quad \sup_M\varphi_0=0.
    \]
\end{prop}

\begin{proof}
    Recall that the canonical Laplacian $\Delta_g$ of $\omega$ is a second-order elliptic operator whose kernel consists of constant functions. By standard linear PDE theory (cf. \cite[Appendix~A]{alesker2013uniform}), there exists a Green function $G$ of $\Delta_g$ such that $G(x,y)>-C$, $||G(x,\cdot)||_{L^1(M,g)}\le C$ for a constant $C>0$ depending only on $M,\omega,J$, and
    \begin{equation}\label{equ:4.1}
        \varphi_0(x)=\frac{1}{\int_M\omega^2}\int_M\varphi_0\omega^2+\int_M\Delta_g\varphi_0(y)G(x,y)\omega^2(y),
    \end{equation}
    for all smooth functions $\varphi_0$ and $\forall x\in M$. Note that
    \begin{equation}\label{equ:4.2}
        \int_M\Delta_g\varphi_0\omega^2=-2\int_M\omega\wedge(dJd\varphi_0+da(\varphi_0))=0.
    \end{equation}
    As a result, by adding \cref{equ:4.1} and a large enough multiple of \cref{equ:4.2}, we can make $G$ positive. Since $\mathcal{D}_J^+(\varphi)>-\omega$, by \cref{eq:4.3} with $s=0$, then
    \[
    \Delta_g\varphi_0=\frac{-2\omega\wedge\mathcal{D}_J^+(\varphi)}{\omega^2}\le2(1-e^{f/2}).
    \]
    It is similar to the proof of \cite[Proposition~3.1]{chu2023fully}, let $x_0\in M$ be a point such that $\varphi_0(x_0)=0$, then using \cref{equ:4.1},
    \begin{align*}
        \int_M(-\varphi_0)\omega^2&=\int_M\omega^2\int_M\Delta_g\varphi_0(y)G(x_0,y)\omega^2(y)\\
        &\le2\int_M\omega^2\int_MG(x_0,y)\omega^2(y)\\
        &\le C.
    \end{align*}
\end{proof}

\begin{rem}
    1$)$ Chu-Tosatti-Weinkove \cite[Proposition~3.1]{ChuTW19}, Tosatti-Weinkove \cite{TosattiW18}, and Sz\'ekelyhidi \cite{Szekelyhidi18} obtained the $C^0$-estimate for $\varphi_0$, by using the Alexandroff-Bakelman-Pucci maximum principle (cf. \cite{blocki2005uniform} \cite{blocki2011uniform}) in the setting of the (almost) complex Monge-Amp\`ere equation.\\
    2$)$ It is possible to get a $C^\infty$ bound of $\varphi_0$ using generalized $\alpha$ invariant \cite{wang2025energy} through the method from \cite{TosattiWY08}.
\end{rem}

\section*{acknowledgment}
    The first-named author is deeply grateful to his advisor Z. L\"u for his constant support. The authors thank Hongyu Wang for suggesting this problem and for many helpful, insightful, and encouraging discussions, and they also thank Qiang Tan, Xiaokui Yang, and Haisheng Liu for valuable conversations. Finally, the authors are thankful for the referee for the helpful suggestions and questions.

\bibliographystyle{amsplain0}
\bibliography{ref}

@book {Aubin98,
    AUTHOR = {Aubin, Thierry},
     TITLE = {Some nonlinear problems in {R}iemannian geometry},
    SERIES = {Springer Monographs in Mathematics},
 PUBLISHER = {Springer-Verlag, Berlin},
      YEAR = {1998},
     PAGES = {xviii+395},
      ISBN = {3-540-60752-8},
   MRCLASS = {58-02 (35J60 35N99 46N20 53C21 53C25 53C55 58E20)},
  MRNUMBER = {1636569},
MRREVIEWER = {John\ M.\ Lee},
       DOI = {10.1007/978-3-662-13006-3},
       URL = {https://doi.org/10.1007/978-3-662-13006-3},
}

@incollection {Calabi57,
    AUTHOR = {Calabi, Eugenio},
     TITLE = {On {K}\"{a}hler manifolds with vanishing canonical class},
 BOOKTITLE = {Algebraic geometry and topology. {A} symposium in honor of
              {S}. {L}efschetz},
     PAGES = {78--89},
 PUBLISHER = {Princeton Univ. Press, Princeton, NJ},
      YEAR = {1957},
   MRCLASS = {53.3X},
  MRNUMBER = {85583},
MRREVIEWER = {P.\ Dolbeault},
}

@article {Cherrier87,
    AUTHOR = {Cherrier, Pascal},
     TITLE = {\'{E}quations de {M}onge-{A}mp\`ere sur les vari\'{e}t\'{e}s
              hermitiennes compactes},
   JOURNAL = {Bull. Sci. Math. (2)},
  FJOURNAL = {Bulletin des Sciences Math\'{e}matiques. 2e S\'{e}rie},
    VOLUME = {111},
      YEAR = {1987},
    NUMBER = {4},
     PAGES = {343--385},
      ISSN = {0007-4497},
   MRCLASS = {58G30 (32C10 35J60 53C55)},
  MRNUMBER = {921559},
MRREVIEWER = {John\ M.\ Lee},
}

@article {Chu20,
    AUTHOR = {Chu, Jianchun},
     TITLE = {The parabolic {M}onge-{A}mp\`ere equation on compact almost
              {H}ermitian manifolds},
   JOURNAL = {J. Reine Angew. Math.},
  FJOURNAL = {Journal f\"{u}r die Reine und Angewandte Mathematik. [Crelle's
              Journal]},
    VOLUME = {761},
      YEAR = {2020},
     PAGES = {1--24},
      ISSN = {0075-4102,1435-5345},
   MRCLASS = {58J35 (35K96 35R01 53C55)},
  MRNUMBER = {4080243},
MRREVIEWER = {Julian\ Scheuer},
       DOI = {10.1515/crelle-2018-0019},
       URL = {https://doi.org/10.1515/crelle-2018-0019},
}

@article {ChuTW19,
    AUTHOR = {Chu, Jianchun and Tosatti, Valentino and Weinkove, Ben},
     TITLE = {The {M}onge-{A}mp\`ere equation for non-integrable almost
              complex structures},
   JOURNAL = {J. Eur. Math. Soc. (JEMS)},
  FJOURNAL = {Journal of the European Mathematical Society (JEMS)},
    VOLUME = {21},
      YEAR = {2019},
    NUMBER = {7},
     PAGES = {1949--1984},
      ISSN = {1435-9855,1435-9863},
   MRCLASS = {32W20 (32Q15 32Q20 32Q60 32U05 35J96 53C15)},
  MRNUMBER = {3959856},
MRREVIEWER = {S\l awomir\ Ko\l odziej},
       DOI = {10.4171/JEMS/878},
       URL = {https://doi.org/10.4171/JEMS/878},
}

@article {DLZ10,
    AUTHOR = {Draghici, Tedi and Li, Tian-Jun and Zhang, Weiyi},
     TITLE = {Symplectic forms and cohomology decomposition of almost
              complex four-manifolds},
   JOURNAL = {Int. Math. Res. Not. IMRN},
  FJOURNAL = {International Mathematics Research Notices. IMRN},
      YEAR = {2010},
    NUMBER = {1},
     PAGES = {1--17},
      ISSN = {1073-7928,1687-0247},
   MRCLASS = {53C15 (53C55)},
  MRNUMBER = {2576281},
MRREVIEWER = {Adriano\ Tomassini},
       DOI = {10.1093/imrn/rnp113},
       URL = {https://doi.org/10.1093/imrn/rnp113},
}

@article {Delanoe96,
    AUTHOR = {Delano\"{e}, Philippe},
     TITLE = {Sur l'analogue presque-complexe de l'\'{e}quation de
              {C}alabi-{Y}au},
   JOURNAL = {Osaka J. Math.},
  FJOURNAL = {Osaka Journal of Mathematics},
    VOLUME = {33},
      YEAR = {1996},
    NUMBER = {4},
     PAGES = {829--846},
      ISSN = {0030-6126},
   MRCLASS = {32F07 (32C17 35J60 53C15)},
  MRNUMBER = {1435456},
MRREVIEWER = {John\ M.\ Lee},
       URL = {http://projecteuclid.org/euclid.ojm/1200787219},
}

@inproceedings {Delanoe97,
    AUTHOR = {Delano\"{e}, Philippe},
     TITLE = {An {$L^\infty$} a priori estimate for the {C}alabi-{Y}au
              operator on compact almost-{K}\"{a}hler manifolds},
 BOOKTITLE = {Proceedings of the {T}hird {I}nternational {W}orkshop on
              {D}ifferential {G}eometry and its {A}pplications and the
              {F}irst {G}erman-{R}omanian {S}eminar on {G}eometry ({S}ibiu,
              1997)},
   JOURNAL = {Gen. Math.},
  FJOURNAL = {General Mathematics},
    VOLUME = {5},
      YEAR = {1997},
     PAGES = {145--149},
      ISSN = {1221-5023,1584-3289},
   MRCLASS = {58J60 (53C21)},
  MRNUMBER = {1723603},
}

@article {FuYau08,
    AUTHOR = {Fu, Ji-Xiang and Yau, Shing-Tung},
     TITLE = {The theory of superstring with flux on non-{K}\"{a}hler
              manifolds and the complex {M}onge-{A}mp\`ere equation},
   JOURNAL = {J. Differential Geom.},
  FJOURNAL = {Journal of Differential Geometry},
    VOLUME = {78},
      YEAR = {2008},
    NUMBER = {3},
     PAGES = {369--428},
      ISSN = {0022-040X,1945-743X},
   MRCLASS = {32J81 (32L81 32Q20 32W20 81T30)},
  MRNUMBER = {2396248},
MRREVIEWER = {Julien\ Keller},
       URL = {http://projecteuclid.org/euclid.jdg/1207834550},
}

@article {HarveyLawson15,
    AUTHOR = {Harvey, F. Reese and Lawson, H. Blaine},
     TITLE = {Potential theory on almost complex manifolds},
   JOURNAL = {Ann. Inst. Fourier (Grenoble)},
  FJOURNAL = {Universit\'{e} de Grenoble. Annales de l'Institut Fourier},
    VOLUME = {65},
      YEAR = {2015},
    NUMBER = {1},
     PAGES = {171--210},
      ISSN = {0373-0956,1777-5310},
   MRCLASS = {32Q60 (31C45 32U05 35D40 35J57 35J96)},
  MRNUMBER = {3449151},
MRREVIEWER = {Yanir\ A.\ Rubinstein},
       URL = {http://aif.cedram.org/item?id=AIF_2015__65_1_171_0},
}

@article {Lejmi10E,
    AUTHOR = {Lejmi, Mehdi},
     TITLE = {Extremal almost-{K}\"{a}hler metrics},
   JOURNAL = {Internat. J. Math.},
  FJOURNAL = {International Journal of Mathematics},
    VOLUME = {21},
      YEAR = {2010},
    NUMBER = {12},
     PAGES = {1639--1662},
      ISSN = {0129-167X,1793-6519},
   MRCLASS = {53C55 (53C15 53D20 58E11)},
  MRNUMBER = {2747965},
MRREVIEWER = {Mario\ Garc\'{\i}a-Fern\'{a}ndez},
       DOI = {10.1142/S0129167X10006690},
       URL = {https://doi.org/10.1142/S0129167X10006690},
}

@article {Lejmi10S,
    AUTHOR = {Lejmi, Mehdi},
     TITLE = {Stability under deformations of extremal almost-{K}\"{a}hler
              metrics in dimension 4},
   JOURNAL = {Math. Res. Lett.},
  FJOURNAL = {Mathematical Research Letters},
    VOLUME = {17},
      YEAR = {2010},
    NUMBER = {4},
     PAGES = {601--612},
      ISSN = {1073-2780},
   MRCLASS = {53C25 (32Q60 53D35 58E11)},
  MRNUMBER = {2661166},
MRREVIEWER = {Yanir\ A.\ Rubinstein},
       DOI = {10.4310/MRL.2010.v17.n4.a2},
       URL = {https://doi.org/10.4310/MRL.2010.v17.n4.a2},
}

@article {LiZhang09,
    AUTHOR = {Li, Tian-Jun and Zhang, Weiyi},
     TITLE = {Comparing tamed and compatible symplectic cones and
              cohomological properties of almost complex manifolds},
   JOURNAL = {Comm. Anal. Geom.},
  FJOURNAL = {Communications in Analysis and Geometry},
    VOLUME = {17},
      YEAR = {2009},
    NUMBER = {4},
     PAGES = {651--683},
      ISSN = {1019-8385,1944-9992},
   MRCLASS = {32Q60 (53D35)},
  MRNUMBER = {2601348},
MRREVIEWER = {Tedi\ C.\ Draghici},
       DOI = {10.4310/CAG.2009.v17.n4.a4},
       URL = {https://doi.org/10.4310/CAG.2009.v17.n4.a4},
}

@article {Szekelyhidi18,
    AUTHOR = {Szekelyhidi, G\'{a}bor},
     TITLE = {Fully non-linear elliptic equations on compact {H}ermitian
              manifolds},
   JOURNAL = {J. Differential Geom.},
  FJOURNAL = {Journal of Differential Geometry},
    VOLUME = {109},
      YEAR = {2018},
    NUMBER = {2},
     PAGES = {337--378},
      ISSN = {0022-040X,1945-743X},
   MRCLASS = {58J05 (32W50 35J60 35J96 35R01 53C55)},
  MRNUMBER = {3807322},
MRREVIEWER = {Bianca\ Santoro},
       DOI = {10.4310/jdg/1527040875},
       URL = {https://doi.org/10.4310/jdg/1527040875},
}

@article {TanWZZ15,
    AUTHOR = {Tan, Qiang and Wang, Hongyu and Zhang, Ying and Zhu, Peng},
     TITLE = {On cohomology of almost complex 4-manifolds},
   JOURNAL = {J. Geom. Anal.},
  FJOURNAL = {Journal of Geometric Analysis},
    VOLUME = {25},
      YEAR = {2015},
    NUMBER = {3},
     PAGES = {1431--1443},
      ISSN = {1050-6926,1559-002X},
   MRCLASS = {53C15 (32Q60 53C55)},
  MRNUMBER = {3358058},
MRREVIEWER = {Manuel\ de Le\'{o}n},
       DOI = {10.1007/s12220-014-9477-2},
       URL = {https://doi.org/10.1007/s12220-014-9477-2},
}

@article {TanWZZ22,
    AUTHOR = {Tan, Qiang and Wang, Hongyu and Zhou, Jiuru and Zhu, Peng},
     TITLE = {On tamed almost complex four-manifolds},
   JOURNAL = {Peking Math. J.},
  FJOURNAL = {Peking Mathematical Journal},
    VOLUME = {5},
      YEAR = {2022},
    NUMBER = {1},
     PAGES = {37--152},
      ISSN = {2096-6075,2524-7182},
   MRCLASS = {53D35 (53C56 53C65)},
  MRNUMBER = {4389489},
MRREVIEWER = {Richard\ Keith\ Hind},
       DOI = {10.1007/s42543-021-00045-7},
       URL = {https://doi.org/10.1007/s42543-021-00045-7},
}

@article {Tian87,
    AUTHOR = {Tian, Gang},
     TITLE = {On {K}\"ahler-{E}instein metrics on certain {K}\"ahler
              manifolds with {$C_1(M)>0$}},
   JOURNAL = {Invent. Math.},
  FJOURNAL = {Inventiones Mathematicae},
    VOLUME = {89},
      YEAR = {1987},
    NUMBER = {2},
     PAGES = {225--246},
      ISSN = {0020-9910,1432-1297},
   MRCLASS = {53C25 (32C10 53C55 58E20)},
  MRNUMBER = {894378},
MRREVIEWER = {M.\ Kalka},
       DOI = {10.1007/BF01389077},
       URL = {https://doi.org/10.1007/BF01389077},
}

@article {TosattiW10,
    AUTHOR = {Tosatti, Valentino and Weinkove, Ben},
     TITLE = {The complex {M}onge-{A}mp\`ere equation on compact {H}ermitian
              manifolds},
   JOURNAL = {J. Amer. Math. Soc.},
  FJOURNAL = {Journal of the American Mathematical Society},
    VOLUME = {23},
      YEAR = {2010},
    NUMBER = {4},
     PAGES = {1187--1195},
      ISSN = {0894-0347,1088-6834},
   MRCLASS = {32W20 (32Q15)},
  MRNUMBER = {2669712},
MRREVIEWER = {S\l awomir\ Ko\l odziej},
       DOI = {10.1090/S0894-0347-2010-00673-X},
       URL = {https://doi.org/10.1090/S0894-0347-2010-00673-X},
}

@incollection {TosattiW18,
    AUTHOR = {Tosatti, Valentino and Weinkove, Ben},
     TITLE = {The {A}leksandrov-{B}akelman-{P}ucci estimate and the
              {C}alabi-{Y}au equation},
 BOOKTITLE = {Nonlinear analysis in geometry and applied mathematics. {P}art
              2},
    SERIES = {Harv. Univ. Cent. Math. Sci. Appl. Ser. Math.},
    VOLUME = {2},
     PAGES = {147--158},
 PUBLISHER = {Int. Press, Somerville, MA},
      YEAR = {2018},
   MRCLASS = {32W50 (32Q99 35R01 53C55)},
  MRNUMBER = {3823885},
MRREVIEWER = {Bianca Santoro},
}

@article {TosattiWY08,
    AUTHOR = {Tosatti, Valentino and Weinkove, Ben and Yau, Shing-Tung},
     TITLE = {Taming symplectic forms and the {C}alabi-{Y}au equation},
   JOURNAL = {Proc. Lond. Math. Soc. (3)},
  FJOURNAL = {Proceedings of the London Mathematical Society. Third Series},
    VOLUME = {97},
      YEAR = {2008},
    NUMBER = {2},
     PAGES = {401--424},
      ISSN = {0024-6115,1460-244X},
   MRCLASS = {32Q25 (53C15)},
  MRNUMBER = {2439667},
MRREVIEWER = {Yanir\ A.\ Rubinstein},
       DOI = {10.1112/plms/pdn008},
       URL = {https://doi.org/10.1112/plms/pdn008},
}

@article {TWWY15,
    AUTHOR = {Tosatti, Valentino and Wang, Yu and Weinkove, Ben and Yang,
              Xiaokui},
     TITLE = {{$C^{2,\alpha}$} estimates for nonlinear elliptic equations in
              complex and almost complex geometry},
   JOURNAL = {Calc. Var. Partial Differential Equations},
  FJOURNAL = {Calculus of Variations and Partial Differential Equations},
    VOLUME = {54},
      YEAR = {2015},
    NUMBER = {1},
     PAGES = {431--453},
      ISSN = {0944-2669,1432-0835},
   MRCLASS = {32W20 (32Q15 32Q60 35B65 35J60 53C55 58J05)},
  MRNUMBER = {3385166},
MRREVIEWER = {Hans-Joachim\ Hein},
       DOI = {10.1007/s00526-014-0791-0},
       URL = {https://doi.org/10.1007/s00526-014-0791-0},
}

@article {WangZ10,
    AUTHOR = {Wang, Hongyu and Zhu, Peng},
     TITLE = {On a generalized {C}alabi-{Y}au equation},
   JOURNAL = {Ann. Inst. Fourier (Grenoble)},
  FJOURNAL = {Universit\'{e} de Grenoble. Annales de l'Institut Fourier},
    VOLUME = {60},
      YEAR = {2010},
    NUMBER = {5},
     PAGES = {1595--1615},
      ISSN = {0373-0956,1777-5310},
   MRCLASS = {53C07 (32Q25 53C15)},
  MRNUMBER = {2766224},
MRREVIEWER = {Anna\ M.\ Fino},
       URL = {http://aif.cedram.org/item?id=AIF_2010__60_5_1595_0},
}

@misc{WangWZ23,
      title={On closed almost complex four manifolds},
      author={Wang, Hongyu and Wang, Ken and Zhu, Peng},
      NOTE={arXiv:2305.09213},
}

@article {Weinkove07,
    AUTHOR = {Weinkove, Ben},
     TITLE = {The {C}alabi-{Y}au equation on almost-{K}\"{a}hler
              four-manifolds},
   JOURNAL = {J. Differential Geom.},
  FJOURNAL = {Journal of Differential Geometry},
    VOLUME = {76},
      YEAR = {2007},
    NUMBER = {2},
     PAGES = {317--349},
      ISSN = {0022-040X,1945-743X},
   MRCLASS = {32Q60 (32W20 53C15)},
  MRNUMBER = {2330417},
MRREVIEWER = {Julien\ Keller},
       URL = {http://projecteuclid.org/euclid.jdg/1180135681},
}

@article {Yau77,
    AUTHOR = {Yau, Shing Tung},
     TITLE = {Calabi's conjecture and some new results in algebraic
              geometry},
   JOURNAL = {Proc. Nat. Acad. Sci. U.S.A.},
  FJOURNAL = {Proceedings of the National Academy of Sciences of the United
              States of America},
    VOLUME = {74},
      YEAR = {1977},
    NUMBER = {5},
     PAGES = {1798--1799},
      ISSN = {0027-8424},
   MRCLASS = {53C55 (14M20 32C10)},
  MRNUMBER = {451180},
MRREVIEWER = {Bang-yen\ Chen},
       DOI = {10.1073/pnas.74.5.1798},
       URL = {https://doi.org/10.1073/pnas.74.5.1798},
}

@article {Yau78,
    AUTHOR = {Yau, Shing Tung},
     TITLE = {On the {R}icci curvature of a compact {K}\"{a}hler manifold
              and the complex {M}onge-{A}mp\`ere equation. {I}},
   JOURNAL = {Comm. Pure Appl. Math.},
  FJOURNAL = {Communications on Pure and Applied Mathematics},
    VOLUME = {31},
      YEAR = {1978},
    NUMBER = {3},
     PAGES = {339--411},
      ISSN = {0010-3640,1097-0312},
   MRCLASS = {53C55 (32C10 35J60)},
  MRNUMBER = {480350},
MRREVIEWER = {Robert\ E.\ Greene},
       DOI = {10.1002/cpa.3160310304},
       URL = {https://doi.org/10.1002/cpa.3160310304},
}

@article {alesker2013uniform,
    AUTHOR = {Alesker, Semyon and Shelukhin, Egor},
     TITLE = {On a uniform estimate for the quaternionic {C}alabi problem},
   JOURNAL = {Israel J. Math.},
  FJOURNAL = {Israel Journal of Mathematics},
    VOLUME = {197},
      YEAR = {2013},
    NUMBER = {1},
     PAGES = {309--327},
      ISSN = {0021-2172,1565-8511},
   MRCLASS = {53C56 (32Q25 32U05)},
  MRNUMBER = {3096617},
MRREVIEWER = {Thomas\ Bruun\ Madsen},
       DOI = {10.1007/s11856-013-0003-1},
       URL = {https://doi.org/10.1007/s11856-013-0003-1},
}

@book{szekelyhidi2014introduction,
    AUTHOR = {Sz\'{e}kelyhidi, G\'{a}bor},
     TITLE = {An introduction to extremal {K}\"{a}hler metrics},
    SERIES = {Graduate Studies in Mathematics},
    VOLUME = {152},
 PUBLISHER = {American Mathematical Society, Providence, RI},
      YEAR = {2014},
     PAGES = {xvi+192},
      ISBN = {978-1-4704-1047-6},
   MRCLASS = {53C55 (14L24 32Q20 53C25)},
  MRNUMBER = {3186384},
MRREVIEWER = {Andrew\ Bucki},
       DOI = {10.1090/gsm/152},
       URL = {https://doi.org/10.1090/gsm/152},
}

@book{joyce2007riemannian,
    AUTHOR = {Joyce, Dominic D.},
     TITLE = {Riemannian holonomy groups and calibrated geometry},
    SERIES = {Oxford Graduate Texts in Mathematics},
    VOLUME = {12},
 PUBLISHER = {Oxford University Press, Oxford},
      YEAR = {2007},
     PAGES = {x+303},
      ISBN = {978-0-19-921559-1},
   MRCLASS = {53C38 (53C29)},
  MRNUMBER = {2292510},
MRREVIEWER = {Spiro\ Karigiannis},
}

@article{huang2023fully,
    AUTHOR = {Huang, Liding and Zhang, Jiaogen},
     TITLE = {Fully nonlinear elliptic equations with gradient terms on
              compact almost {H}ermitian manifolds},
   JOURNAL = {Math. Z.},
  FJOURNAL = {Mathematische Zeitschrift},
    VOLUME = {303},
      YEAR = {2023},
    NUMBER = {2},
     PAGES = {Paper No. 36, 24},
      ISSN = {0025-5874,1432-1823},
   MRCLASS = {32W50 (32Q60 32W20 35J96)},
  MRNUMBER = {4530185},
MRREVIEWER = {Ahmet\ Okay\ \c{C}elebi},
       DOI = {10.1007/s00209-022-03202-5},
       URL = {https://doi.org/10.1007/s00209-022-03202-5},
}

@article{guan2014second,
    AUTHOR = {Guan, Bo},
     TITLE = {Second-order estimates and regularity for fully nonlinear
              elliptic equations on {R}iemannian manifolds},
   JOURNAL = {Duke Math. J.},
  FJOURNAL = {Duke Mathematical Journal},
    VOLUME = {163},
      YEAR = {2014},
    NUMBER = {8},
     PAGES = {1491--1524},
      ISSN = {0012-7094,1547-7398},
   MRCLASS = {35R01 (35B65 35J60 35J96 58J05)},
  MRNUMBER = {3284698},
MRREVIEWER = {Pierpaolo\ Esposito},
       DOI = {10.1215/00127094-2713591},
       URL = {https://doi.org/10.1215/00127094-2713591},
}

@article{chu2023fully,
    AUTHOR = {Chu, Jianchun and Huang, Liding and Zhang, Jiaogen},
     TITLE = {Fully non-linear elliptic equations on compact almost
              {H}ermitian manifolds},
   JOURNAL = {Calc. Var. Partial Differential Equations},
  FJOURNAL = {Calculus of Variations and Partial Differential Equations},
    VOLUME = {62},
      YEAR = {2023},
    NUMBER = {3},
     PAGES = {Paper No. 105, 34},
      ISSN = {0944-2669,1432-0835},
   MRCLASS = {58J05 (32Q60 35J60)},
  MRNUMBER = {4553956},
MRREVIEWER = {Weisong\ Dong},
       DOI = {10.1007/s00526-023-02445-2},
       URL = {https://doi.org/10.1007/s00526-023-02445-2},
}

@article{blocki2005uniform,
    AUTHOR = {B\l ocki, Zbigniew},
     TITLE = {On uniform estimate in {C}alabi-{Y}au theorem},
   JOURNAL = {Sci. China Ser. A},
  FJOURNAL = {Science in China. Series A. Mathematics},
    VOLUME = {48},
      YEAR = {2005},
    NUMBER = {suppl.},
     PAGES = {244--247},
      ISSN = {1006-9283,1862-2763},
   MRCLASS = {32W20 (32Q25)},
  MRNUMBER = {2156505},
MRREVIEWER = {S\l awomir\ Ko\l odziej},
       DOI = {10.1007/BF02884710},
       URL = {https://doi.org/10.1007/BF02884710},
}

@article{blocki2011uniform,
     AUTHOR = {B\l ocki, Zbigniew},
     TITLE = {On the uniform estimate in the {C}alabi-{Y}au theorem, {II}},
   JOURNAL = {Sci. China Math.},
  FJOURNAL = {Science China. Mathematics},
    VOLUME = {54},
      YEAR = {2011},
    NUMBER = {7},
     PAGES = {1375--1377},
      ISSN = {1674-7283,1869-1862},
   MRCLASS = {32Q25 (32W20 53C55)},
  MRNUMBER = {2817572},
MRREVIEWER = {Bianca\ Santoro},
       DOI = {10.1007/s11425-011-4197-6},
       URL = {https://doi.org/10.1007/s11425-011-4197-6},
}

@article {Tosatti2019complex,
    AUTHOR = {Tosatti, Valentino and Weinkove, Ben},
     TITLE = {The complex {M}onge-{A}mp\`ere equation with a gradient term},
   JOURNAL = {Pure Appl. Math. Q.},
  FJOURNAL = {Pure and Applied Mathematics Quarterly},
    VOLUME = {17},
      YEAR = {2021},
    NUMBER = {3},
     PAGES = {1005--1024},
      ISSN = {1558-8599},
   MRCLASS = {32W20 (53C55)},
  MRNUMBER = {4278956},
MRREVIEWER = {Sibel \c{S}ahin},
       DOI = {10.4310/PAMQ.2021.v17.n3.a7},
       URL = {https://doi.org/10.4310/PAMQ.2021.v17.n3.a7},
}

@article{wang2025energy,
      title={Energy functionals on almost {K}\"ahler manifolds: {I}},
      author={Wang, Ken and Zhang, Zuyi and Zhou, Jiuru},
      NOTE={arXiv:2511.17086},
}

\end{document}